\documentclass[draft]{amsart}
\usepackage{amssymb} % SRC Specials: DVI [Inverse] Search
\newtheorem{thm}{Theorem}[section]
\newtheorem{cor}[thm]{Corollary}

\newtheorem{pro}[thm]{Proposition}
\newtheorem{DEF}[thm]{Definition}
\newtheorem{exa}[thm]{Example}

\newtheorem{rem}[thm]{Remark}
\numberwithin{equation}{section}
\begin{document}
\title[]{ SEGAL FR$\acute{E}$CHET ALGEBRAS}%
\author{ F. Abtahi, S. Rahnama and A. Rejali}%
%\address{Department of Mathematics, University of Isfahan,
%Isfahan
%81746-73441, Iran}%

%\address{Islamic Azad University Mobarakeh Branch, Iran}%

%
%\email{f.abtahi@sci.ui.ac.ir}%

%\email{a.yousofzade@sci.ui.ac.ir}

\thanks{}%

%\subjclass



\date{}%
%\dedicatory{}%
%\commby{}%
%------------------------------------------------------------------------
%\baselineskip=1cm

\begin{abstract}
Let $({\mathcal A},p_{\ell})_{\ell\in {\Bbb N}}$ be a
Fr$\acute{e}$chet algebra. In this paper, we introduce the concept
of Segal Fr$\acute{e}$chet algebra and investigate known results
about abstract Segal algebras, for Segal Fr$\acute{e}$chet
algebras. Also we recall the concept of approximate identities for
topological algebras and provide some remarkable results for Segal
Fr$\acute{e}$chet algebras. Moreover, we verify ideal theorem for
Fr$\acute{e}$chet algebras and characterize closed ideals of Segal
Fr$\acute{e}$chet algebra $({\mathcal B},q_{m})_{m\in {\Bbb N}}$
in $({\mathcal A},p_{\ell})_{\ell\in {\Bbb N}}$.
\end{abstract}
\date{}%

\maketitle

\section{\bf Introduction}

Segal algebras were first defined by H. Reiter for group algebras
in \cite{re}. Then Burnham \cite{Burn} introduced the notion of
abstract Segal algebra in Banach algebra $(\mathcal
A,\|.\|_{\mathcal A})$. As an important result, he showed that
approximate identities in the proper abstract Segal algebras can
not be bounded. Also he examined ideal structure of a class of
Banach algebras that subsume Cigler's normed ideals \cite{Cig} and
Reiter's Segal algebras \cite[Page 127]{re1}. Moreover, he proved
that if $(\mathcal B,\|.\|_{\mathcal B})$ is a commutative
abstract Segal algebra in $(\mathcal A,\|.\|_{\mathcal A})$ with
an approximate identity, then there is a one to one correspondence
between the closed ideals of $B$ and those of $A$. We also found
many other valuable results in \cite{Bar}. In fact he considered
the relationship between the Banach algebras $(\mathcal
A,\|.\|_{\mathcal A})$ and $(\mathcal B,\|.\|_{\mathcal B})$,
whenever $\mathcal B$ is also a left ideal in $\mathcal A$. For
example, he proved that if $\|b\|_{\mathcal A}\leq
D\|b\|_{\mathcal B}$, for all $b\in\mathcal B$ and some $D>0$,
then $\mathcal B$ is a Banach left $\mathcal A-$module
\cite[Theorem 2.3]{Bar}. Moreover, he showed that if there is the
certain above inequality between $\|.\|_{\mathcal A}$ and
$\|.\|_{\mathcal B}$, and also $\mathcal B$ is a Banach left
$\mathcal A-$module, then $\mathcal B$ is a left ideal in
$cl_{\mathcal A}(\mathcal B)$, closure $\mathcal B$ in $\mathcal
A$ \cite[Proposition 2.1]{Bar}. In fact these observations lead to
summarize the definition of abstract Segal algebras. We also refer
to \cite{Burn2}, \cite{Burn3}, \cite{Cig}, \cite{Dit}, \cite{D},
\cite{f2}, \cite{GSR} and \cite{Ri}, which contain valuable
results related to this subject.

Some of the notions related to Banach algebras, have been
introduced and studied for Fr$\acute{e}$chet algebras. For
example, the notion of amenability of a Fr$\acute{e}$chet algebra
was introduced by A. Yu. Pirkovskii \cite{Pir}. He generalized
some theorems about amenability of Banach algebras such as
strictly flat Banach $A$-bimodule, virtual diagonal and
approximate diagonal of Banach algebras, to Fr$\acute{e}$chet
algebras. Also in \cite{Law}, P. Lawson and C. J. Read introduced
and studied some notions about approximate amenability and
approximate contractibility of Fr$\acute{e}$chet algebras.

The present work is essentially raised from the available results
in the field of abstract Segal algebras. Let $(\mathcal
A,p_{\ell})_{\ell\in {\Bbb N}}$ be a Fr$\acute{e}$chet algebra.
According to the definition of abstract Segal algebras
\cite{Burn}, we first introduce the concept of Segal
Fr$\acute{e}$chet algebra in $(\mathcal A,p_{\ell})_{\ell\in {\Bbb
N}}$. To provide an example about Segal Fr$\acute{e}$chet algebra,
we verify the results of Burnham and Barnes for Segal
Fr$\acute{e}$chet algebras and show that Proposition 2.1 and
Theorem 2.3 of \cite{Bar} are valid for Segal Fr$\acute{e}$chet
algebras, as well. It follows that as in abstract Segal algebras,
the definition of Segal Fr$\acute{e}$chet algebra can be
summarized. Moreover, we recall the concept of approximate units,
multiple approximate identity and approximate identity for
Fr$\acute{e}$chet algebras and investigate \cite[Theorem
1.2]{Burn} for Segal Fr$\acute{e}$chet algebras and obtain the
same result. Indeed, we prove that if $(\mathcal B,q_{m})_{m\in
{\Bbb N}}$ is a proper Segal Fr$\acute{e}$chet algebra in
$(\mathcal A,p_{\ell})_{\ell\in {\Bbb N}}$ which contains an
approximate identity $(e_{\alpha})$, then $(e_{\alpha})$ can not
be bounded in $(\mathcal B,q_{m})_{m\in {\Bbb N}}$. At the end, we
verify ideal theorem and characterize closed left ideals of a
Segal Fr$\acute{e}$chet algebra. In fact we show that every closed
left ideal of any Segal Fr$\acute{e}$chet algebra in $(\mathcal
A,p_{\ell})_{\ell\in {\Bbb N}}$, is the intersection of a closed
left ideal of $\mathcal A$ with $\mathcal B$.

\section{\bf Preliminaries}

In this section, we present some basic definitions related to
Fr$\acute{e}$chet algebras, which will be required throughout the
paper. See \cite{Gold}, \cite{Hel2} and \cite{ME}, for more
information.

A locally convex topological vector space $E$ is a topological
vector space in which the origin has a local base of absolutely
convex absorbent sets. A collection $\mathcal{U}$ of zero
neighborhoods in $E$ is called a fundamental system of zero
neighborhoods, if for every zero neighborhood $U$, there exists a
$V\in {\mathcal U}$ and an $\varepsilon>0$ such that $\varepsilon
V\subset U$. Throughout the paper, all locally convex spaces are
assumed to be Hausdorff. $S\subseteq E$ is called bounded if for
every zero neighborhood $U$, there exists scalar $\lambda$ such
that $S\subseteq\lambda U$; it is called balanced if for each
$\alpha\in\Bbb C$ with $|\alpha|\leq 1$, $\alpha S\subseteq S$.
Moreover $S$ is called absorbing if for each $x\in E$, there is
the scalar $\lambda$ such that $x\in\lambda S$.

A family $(p_\alpha)_{\alpha\in A}$ of continuous seminorms on $E$
is called a fundamental system of seminorms, if the sets
$$U_{\alpha}=\{x\in E :p_{\alpha}(x)<1\} \;\;\;\;\;\;(\alpha\in A)$$
form a fundamental system of zero neighborhoods. We refer to
\cite[page 251]{ME}, for more details. Every Hausdorff locally
convex space $E$ has a fundamental system of seminorms
$(p_{\alpha})_{\alpha\in A}$; equivalently a family of the
seminorms satisfying the following properties:
\begin{enumerate}
\item[(i)] For every $x\in E$ with $x\neq 0$, there exists an
${\alpha}\in A$ with $p_{\alpha}(x)>0$; \item[(ii)] For all
${\alpha,\beta}\in A$, there exist ${\gamma}\in A$ and $C>0$ such
that $$\max(p_{\alpha}(x),p_{\beta}(x))\leq C
p_{\gamma}(x)\;\;\;\;\;\;\;\;(x\in E);$$
\end{enumerate}
see \cite[Lemmas 22.4,22.5]{ME}.

Now let $E$ be a locally convex space and
$(p_{\alpha})_{\alpha\in\Lambda}$ be a fundamental system of
seminorms. A subset $B$ of $E$ is bounded if and only if
$\sup_{x\in B}p_{\alpha}(x)<\infty$, for each $\alpha\in\Lambda$.

We recall \cite[Proposition 22.6]{ME}, which is very useful in our
later discussions.

\begin{pro}\label{p2}
Let $E$ and $F$ be locally convex spaces with the fundamental
system of seminorms $(p_{\alpha})_{\alpha \in A}$ in $E$ and
$(q_{\beta})_{\beta \in B}$ in $F$. Then for every linear mapping
$T: E\longrightarrow F$, the following assertions are equivalent.
\begin{enumerate}
\item[(i)]  $T$ is continuous. \item[(ii)] $T$ is continuous at
$0$. \item[(iii)] For each $\beta\in B$ there exist an $\alpha\in
A$ and $C>0$, such that $$q_{\beta}(T(x))\leq Cp_{\alpha}(x),$$
for all $x\in E$.
\end{enumerate}
\end{pro}

\noindent It should be noted that by \cite[page 24]{Hel2}, if
$(E,p_{\mu})$, $(F,q_{\lambda})$ and $(G,r_{\nu})$ are locally
convex spaces, and $\theta:E\times F\to G$ is a bilinear map, then
$\theta$ is jointly continuous if and only if for any ${\nu_0}$
there exist ${\mu_0}$ and ${\lambda_0}$ such that the bilinear map
$$
\theta:(E,p_{\mu_0})\times (F,q_{\lambda_0})\longrightarrow
(G,r_{\nu_0})
$$
is jointly continuous. In other words there exists $C>0$ such that
$$
r_{\nu_0}(\theta(x,y))\leq Cp_{\mu_0}(x)q_{\lambda_0}(y),
$$
for all $x\in E$ and $y\in F$. Recall from \cite{Shi} that
bilinear map $f$ from $E\times F$ into $G$ is said to be
separately continuous if all partial maps $f_x:F\rightarrow G$ and
$f_y:E\rightarrow G$ defined by $y\mapsto f(x,y)$ and $x\mapsto
f(x,y)$, respectively, are continuous for each $x\in E$ and $y\in
F$. By \cite[chapter.III.5.1]{Shi}, both types of continuity
coincide in the class of Fr$\acute{e}$chet spaces and in
particular, Banach spaces. In such a situation, we use only
"continuous" phrase. A topological algebra ${\mathcal A}$ is an
algebra, which is a topological vector space and the
multiplication ${\mathcal A}\times{\mathcal A}\longrightarrow
{\mathcal A}$, defined by $(a,b)\mapsto ab$ is a jointly
continuous mapping; see \cite[Definition (3.1.5)]{Gold}. A locally
m-convex Fr$\acute{e}$chet algebra (lmc Fr$\acute{e}$chet algebra)
is a complete topological algebra, whose topology is given by a
countable family of increasing submultiplicative seminorms; see
\cite{Gold} and \cite{Hel3} for more information. For convenience,
throughout the paper, we use the label "Fr$\acute{e}$chet
algebra", instead of "lmc Fr$\acute{e}$chet algebra".

\section{\bf Introduction of Segal Fr$\acute{e}$chet algebra}

A Banach space $({\mathcal B},\|.\|_{\mathcal B})$ is called an
abstract Segal algebra of $({\mathcal A},\|.\|_{\mathcal A})$ if
the following statements are satisfied:
\begin{enumerate}
\item[$(A_1)$] ${\mathcal B}$ is a dense left ideal in ${\mathcal
A}$. \item[$(A_2)$] There exists $M>0$ such that $\|f\|_{\mathcal
A}\leq M\|f\|_{\mathcal B}$, for each $f\in {\mathcal B}$.
\item[$(A_3)$] There exists $C>0$ such that $\|fg\|_{\mathcal
B}\leq C\|f\|_{\mathcal A}\|g\|_{\mathcal B}$, for each $f,g\in
{\mathcal B}$.
\end{enumerate}
Equivalently, $({\mathcal B},\|.\|_{\mathcal B})$ is an abstract
Segal algebra of $({\mathcal A},\|.\|_{\mathcal A})$ if it is
continuously embedded in $\mathcal A$ and also it is a Banach left
$\mathcal A-$module. Retrieved from this definition, we introduce
the concept of Segal Fr$\acute{e}$chet algebra as the following.

\begin{DEF}\label{d1} \rm
A Fr$\acute{e}$chet algebra $(\mathcal B,q_m)_{m\in \mathbb{N}}$
is a Segal Fr$\acute{e}$chet algebra in a Fr$\acute{e}$chet
algebra $(\mathcal A, p_{\ell})_{{\ell}\in \mathbb{N}}$ if the
following conditions are satisfied:
\begin{enumerate}
\item[(i)] $\mathcal B$ is a dense left ideal in $\mathcal A$.
\item[(ii)] The map
\begin{equation}\label{e1}
i:(\mathcal B, q_m)_{m\in {\Bbb N}}\longrightarrow (\mathcal A ,
p_{\ell})_{\ell\in {\Bbb N}},\;\;\;a \mapsto  a, \;\;\;\;(a\in
\mathcal B)
\end{equation}
is continuous. \item[(iii)] The map
\begin{equation}\label{e2}
(\mathcal B, p_{\ell})_{\ell\in {\Bbb N}}\times (\mathcal B,
q_m)_{m\in {\Bbb N}}\longrightarrow (\mathcal B, q_m)_{m\in {\Bbb
N}},\;\;\;(a,b)\mapsto ab,\;\;\;(a,b\in \mathcal B)
\end{equation}
is jointly continuous.
\end{enumerate}
$\mathcal B$ is called a symmetric Segal Fr$\acute{e}$chet algebra
in $\mathcal A$, if $\mathcal B$ is a dense two-sided ideal in
$\mathcal A$ and \eqref{e1} and \eqref{e2} hold. Moreover the map
$$(\mathcal B, q_m)_{m\in {\Bbb N}}\times (\mathcal B, p_{\ell})_{\ell\in {\Bbb N}}\longrightarrow
(\mathcal B, q_m)_{m\in {\Bbb N}},\;\;\;(a,b)\mapsto
ab,\;\;\;(a,b\in \mathcal B)$$ is jointly continuous.
\end{DEF}

Note that the concept of Segal Fr$\acute{e}$chet algebra is
coincided to the concept of abstract Segal algebras, in the case
where $\mathcal A$ and $\mathcal B$ are Banach algebras.\\

For presenting an example about Definition \ref{d1}, we shall make
some preparations.

\begin{rem}\label{r1}\rm
Let $(\mathcal B,q_m)_{m\in {\Bbb N}}$ be a Segal
Fr$\acute{e}$chet algebra in Fr$\acute{e}$chet algebra $(\mathcal
A, p_{\ell})_{\ell\in {\Bbb N}}$.
\begin{enumerate}
\item As in Banach algebras, conditions (ii) and (iii) in
Definition \ref{d1}, can be given similar to $(A_2)$ and $(A_3)$
in the definition of abstract Segal algebras. In fact Proposition
\ref{p2} implies that condition (ii) in Definition \ref{d1} is
equivalent to the fact that, for every $\ell\in\mathbb{N}$, there
exist $M_{\ell}>0$ and $m_{\ell}\in\mathbb{N}$ such that,
$p_{\ell}(b)\leq M_{\ell} q_{m_{\ell}}(b)$, for all $b\in\mathcal
B$. Moreover by \cite[Page 24]{Hel2}, continuity of the map in
(iii) is equivalent to the fact that for every $m\in\mathbb{N}$
there exist $K_m>0$ and $\ell_{m}, n_m\in\mathbb{N}$ such that for
all $a,b\in\mathcal B$,
$$q_m(ab)\leq K_m p_{\ell_{m}}(a)q_{n_m}(b).$$
\item As in abstract Segal algebras, since ${\mathcal B}$ is a
dense subspace of ${\mathcal A}$, we can extend the continuous map
\eqref{e2} to the unique continuous map
$$(\mathcal A, p_{\ell})_{\ell\in {\Bbb N}}\times(\mathcal
B,q_m)_{m\in {\Bbb N}}\longrightarrow(\mathcal B,q_m)_{m\in {\Bbb
N}},\;\;\;\;\;\;(a,b)\mapsto ab.$$ Indeed, since
$$(\mathcal B,p_{\ell})_{\ell\in {\Bbb N}}\times(\mathcal
B,q_m)_{m\in {\Bbb N}}\longrightarrow(\mathcal B,q_m)_{m\in {\Bbb
N}},\;\;\;(a,b)\mapsto  ab,$$ is jointly continuous, so it is
separately continuous and consequently both maps
$$(\mathcal B,p_{\ell})_{\ell\in {\Bbb N}}\longrightarrow(\mathcal B,q_m)_{m\in {\Bbb N}},\;\;\;a\mapsto  ab,$$
and
$$(\mathcal B,q_m)_{m\in {\Bbb N}}\longrightarrow(\mathcal B,q_m)_{m\in {\Bbb N}},\;\;\;b\mapsto  ab,$$
are continuous, for all $a,b\in\mathcal B$. It is not hard to see
that for each $a\in\mathcal A$, the map
$$(\mathcal B,q_m)_{m\in {\Bbb N}}\longrightarrow(\mathcal B,q_m)_{m\in {\Bbb N}},\;\;\;b\mapsto  ab,$$
is continuous. Moreover by \cite[Lemma 22.19]{ME}, the map
$$(\mathcal B,p_{\ell})_{\ell\in {\Bbb N}}\longrightarrow(\mathcal B,q_m)_{m\in
{\Bbb N}},\;\;\;a\mapsto  ab,$$ has a unique extension to the continuous map
$$(\mathcal A,p_{\ell})_{\ell\in {\Bbb N}}\longrightarrow(\mathcal B,q_m)_{m\in
{\Bbb N}},\;\;\;a\mapsto  ab,$$ for all $b\in\mathcal B$. It
follows that the map
$$(\mathcal A,p_{\ell})_{\ell\in {\Bbb N}}\times(\mathcal B,q_m)_{m\in {\Bbb N}}
\longrightarrow(\mathcal B,q_m)_{m\in {\Bbb N}},$$ is separately
continuous. Since $\mathcal A$ and $\mathcal B$ are
Fr$\acute{e}$chet algebra, thus the map is jointly continuous.
\end{enumerate}
\end{rem}

As the main results of this section, we prove Proposition 2.1 and
Theorem 2.3 of \cite{Bar}, for the Fr$\acute{e}$chet algebras. In
fact we show that the definition of Segal Fr$\acute{e}$chet
algebra can be summarized. First, we recall closed graph theorem
for the Fr$\acute{e}$chet spaces. Let $E$ and $F$ be
Fr$\acute{e}$chet spaces and $T:E\longrightarrow F$ is a linear
mapping such that its graph, $\{(x,T(x)):\;\;x\in E\},$ is closed
in $E\times F$. Then $T$ is continuous; see \cite[B.2]{Gold}.

\begin{thm}\label{t1}
Let $(\mathcal B, q_m)_{m\in {\Bbb N}}$ be a Fr$\acute{e}$chet
algebra, which is a left ideal in a Fr$\acute{e}$chet algebra
$(\mathcal A, p_{\ell})_{\ell\in {\Bbb N}}$. If the map
$$i:(\mathcal B, q_m)_{m\in {\Bbb N}}\longrightarrow (\mathcal A,
p_{\ell})_{\ell\in {\Bbb N}}\;\;\;\;\;\;a\mapsto a$$ is
continuous, then the map
$$(\mathcal A, p_{\ell})_{\ell\in {\Bbb N}}\times (\mathcal B,q_m)_{m\in {\Bbb N}}\longrightarrow
(\mathcal B,q_m)_{m\in {\Bbb N}},\;\;\;(a,b)\mapsto ab$$ is
continuous.
\end{thm}

\begin{proof}
Let $b\in \mathcal B$ and define $$T_b:(\mathcal A,
p_{\ell})_{\ell\in {\Bbb N}} \longrightarrow(\mathcal B,
q_m)_{m\in {\Bbb N}}$$ by $T_b(a)=ab$, for each $a\in \mathcal A.$
By the closed graph theorem in Fr$\acute{e}$chet algebras, we show
that $T_b$ is continuous. Suppose that $(a_n)_{n\in \mathbb{N}}$
is a sequence in $\mathcal A$ such that $\lim_{n\rightarrow
\infty}p_{\ell}(a_n)=0$ and
\begin{equation}\label{e3}
\lim_{n\rightarrow \infty}q_m(T_b(a_n)-c)=0,
\end{equation}
for all ${\ell},m\in\mathbb{N}$ and some $c\in\mathcal B$. We show
that $c=0$. By the hypothesis for each $\ell\in {\Bbb N}$, there
exist $M_{\ell}>0$ and $m_{\ell}\in \mathbb{N}$, such that
$$p_{\ell}(a)\leq M_{\ell} q_{m_{\ell}}(a), \;\;\;\;\;\;\;\;(a\in \mathcal
B)$$ and so for each $n\in\Bbb N$
\begin{equation}\label{e4}
p_{\ell}(a_nb-c)\leq M_{\ell}q_{m_{\ell}}(a_nb-c).
\end{equation}
By \eqref{e3}, the right hand side of the above inequality tends
to zero. The inequality \eqref{e4} implies that
$\lim_{n\rightarrow \infty}a_nb=c$, in the topology of $\mathcal
A$. Moreover, since $p_{\ell}$ is submultiplicative, we have
$p_{\ell}(a_nb)\leq p_{\ell}(a_n)p_ {\ell}(b).$ It follows that
$\lim_{n\rightarrow \infty}a_nb=0,$ in the topology of $\mathcal
A$. Thus $c=0$, as claimed. Consequently $T_b$ is continuous.
Similarly, one can show that for each $a\in {\mathcal A}$, the map
$S_a:(\mathcal B,q_m)_{m\in {\Bbb N}}\to (\mathcal B,q_m)_{m\in
{\Bbb N}}$ defined by $S_a(b)=ab$ is continuous. This completes
the proof.
\end{proof}

\begin{thm}
Let $(\mathcal A, p_{\ell})_{\ell\in {\Bbb N}}$ and $(\mathcal
B,q_m)_{m\in {\Bbb N}}$ be Fr$\acute{e}$chet algebras, such that
$\mathcal B$ is a subalgebra of $\mathcal A$ and both maps given in
\emph{(ii)} and \emph{(iii)} in Definition $\ref{d1}$ are
continuous. Then $\mathcal B$ is a left ideal in $cl_{\mathcal
A}(\mathcal B)$, the closure of $\mathcal B$ in $\mathcal A$.
\end{thm}

\begin{proof}
Suppose that $b\in\mathcal B$ and $a\in cl_{\mathcal A}(\mathcal
B)$. Thus there exists a sequence $(a_n)_{n\in\mathbb{N}}$ in
$\mathcal B$ such that $\lim_{n\rightarrow
\infty}p_{\ell}(a_n-a)=0$, for every $\ell\in\mathbb{N}$. Also by
Remark \ref{r1}, for each $m\in\mathbb{N}$, there exist $C_m>0$
and $n_m, \ell_{m}\in\mathbb{N}$ such that
$$q_m(a_nb-a_kb)=q_m((a_n-a_k)b)\leq C_m p_{\ell_{m}}(a_n-a_k)q_{n_m}(b),$$ for
all $n,k\in\mathbb{N}$. Since $(a_n)_{n\in\mathbb{N}}$ is
convergent in $\mathcal A$, it is a Cauchy sequence in $\mathcal
A$, and so $(a_nb)_{n\in\mathbb{N}}$ is a Cauchy sequence in
$\mathcal B$. Consequently there exists $c\in\mathcal B$ such that
for every $m\in\mathbb{N}$,
$\lim_{n\rightarrow\infty}q_m(a_nb-c)=0$. On the other hand by
Remark \ref{r1}, for every $\ell\in\mathbb{N}$, there exist
$C_{\ell}>0$ and $m_{\ell}\in\mathbb{N}$ such that
$$p_{\ell}(a_nb-c)\leq C_{\ell} q_{m_{\ell}}(a_nb-c).$$ Since the
right hand side of the above inequality tends to zero,
$\lim_{n\rightarrow\infty}a_nb=c,$ in the topology of $\mathcal
A$. Also by the continuity of multiplication in $\mathcal A$,
$\lim_{n\rightarrow\infty}a_nb=ab$. It follows that $ab=c$, and so
$ab\in\mathcal B$. This completes the proof.
\end{proof}

\begin{cor}\label{c1}
Let $(\mathcal A, p_{\ell})_{\ell\in {\Bbb N}}$ and $(\mathcal
B,q_m)_{m\in {\Bbb N}}$ be Fr$\acute{e}$chet algebras, such that
$\mathcal B$ is a dense subalgebra of $\mathcal A$ and  both maps
given in \emph{(ii)} and \emph{(iii)} in Definition $\ref{d1}$ are
continuous. Then $\mathcal B$ is a Segal Fr$\acute{e}$chet
algebra.
\end{cor}

\begin{rem}\label{r2}\rm
Let $(\mathcal A, p_{\ell})_{\ell\in {\Bbb N}}$ and $(\mathcal
B,q_m)_{m\in {\Bbb N}}$ be Fr$\acute{e}$chet algebras. The
definition of Segal Fr$\acute{e}$chet algebra may be summarized as
the following:
\begin{enumerate}
\item By Theorem \ref{t1}, part (iii) in Definition \ref{d1} can
be omitted. In fact a Fr$\acute{e}$chet algebra $(\mathcal
B,q_m)_{m\in \mathbb{N}}$ is a Segal Fr$\acute{e}$chet algebra in
Fr$\acute{e}$chet algebra $(\mathcal A, p_{\ell})_{{\ell}\in
\mathbb{N}}$, if the conditions (i) and (ii) in Definition
\ref{d1} are satisfied. \item If ${\mathcal B}$ is a dense
subalgebra of ${\mathcal A}$, then Corollary \ref{c1} mentions
that in Definition \ref{d1}, condition (i) can be obtained by (ii)
and (iii).
\end{enumerate}
\end{rem}

Now we are in a position to provide examples concerning Segal
Fr$\acute{e}$chet algebras. We found many examples in \cite{Sch},
which satisfy the conditions of Segal Fr$\acute{e}$chet algebras
and so can be good examples of our definition.

We explain here Part $(b)$ of \cite[Example 3.3]{Sch}, which is a
nice example in this field.

\begin{exa}\rm
Let $X$ be a infinite countable set. A function $\ell:
X\rightarrow [1,\infty)$ is a scale on $X$. We say that a scale
$\ell$ on $X$ is proper if the inverse map $\ell^{-1}$ takes
bounded subsets of $[1,\infty)$ to finite subsets of $X$. For the
family of scales $\ell=\{\ell^n\}_{n=0}^{\infty}$ on $X$, define
$$
{\mathcal S}_{\ell}^\infty(X)=\left\{\varphi: X\rightarrow \Bbb
C,\;\; \|\varphi\|_n^\infty<\infty, \forall\;n\in\Bbb N\right\},
$$
where
$$
\|\varphi\|_n^{\infty}=\sup_{x\in
X}\left\{\ell^n(x)|\varphi(x)|\right\}.
$$
Then ${\mathcal S}_{\ell}^{\infty}(X)$ is called the sup-norm
$\ell-$rapidly vanishing functions on $X$. The family
$(\ell^n)_{n=0}^{\infty}$ will satisfy $\ell^0\leq\ell^1\leq . . .
\ell^n\leq . . . $, so that the families of norms
$\{\|.\|_n^{\infty}\}_{n=0}^{\infty}$ are increasing. Moreover, it
is easy to see that all of them are submultiplicative under
pointwise product. In fact ${\mathcal S}_{\ell}^\infty(X)$ is a
Fr$\acute{e}$chet algebra. Now consider $c_0(X)$, the commutative
Banach algebra of complex-valued sequences which vanish at
infinity, with pointwise multiplication and sup-norm
$\|.\|_\infty$. We show that ${\mathcal
S}_{\ell}^\infty(X)\subseteq c_0(X)$. Suppose on the contrary that
there exists $\varphi\in {\mathcal S}_{\ell}^\infty(X)$ such that
$\varphi\not\in c_0(X)$. Thus there is $\varepsilon>0$ such that
for each finite subset $F$ of $X$, there exists $x_F\not\in F$
with $|\varphi(x_F)|\geq\varepsilon$. Since $\ell$ is a proper
scale, thus for each $n\in\Bbb N$, there exists
$x_n\not\in\ell^{-1}([1,n])$ such that
$|\varphi(x_n)|\geq\varepsilon$. It follows that
$$
\sup_{x\in X}\left\{\ell(x)|\varphi(x)|\right\}=\infty,
$$
which contradicts the assumption of $\varphi\in {\mathcal
S}_{\ell}^\infty(X)$. Therefore ${\mathcal
S}_{\ell}^\infty(X)\subseteq c_0(X)$. It is easy to see that the
inequalities
$\|\varphi\;\psi\|_n^{\infty}\leq\|\varphi\|_n^{\infty}\|\psi\|_\infty$
are satisfied, for all $\varphi,\psi\in {\mathcal
S}_{\ell}^\infty(X)$. Since ${\mathcal S}_{\ell}^\infty(X)$
contains the space of finite support functions denoted by
$c_{00}(X)$, it follows that ${\mathcal S}_{\ell}^\infty(X)$ is a
dense Fr$\acute{e}$chet ideal in $c_0(X)$, and so ${\mathcal
S}_{\ell}^\infty(X)$ is a Segal Fr$\acute{e}$chet algebra in
$c_0(X)$.
\end{exa}

\section{\bf Main results}

In this section, we prove some other results of \cite{Bar} and
\cite{Burn} for Segal Fr$\acute{e}$chet algebras. We require
recall the following definitions from \cite{Hel3} and \cite{Pir}.

\begin{DEF}\rm
Let $(\mathcal A, p_{\ell})_{\ell\in \mathbb{N}}$ be a
Fr$\acute{e}$chet algebra.
\begin{enumerate}
\item We say that $\mathcal A$ has left (right) approximate units
if for each $x\in \mathcal A$ and $\ell\in \mathbb{N}$ and
$\varepsilon>0$, there exists $u\in \mathcal A$ such that
$p_{\ell}(ux-x)<\varepsilon\;$ ($p_{\ell}(xu-x)<\varepsilon$).
Moreover we say that $\mathcal A$ has a bounded left (right)
approximate units if there exists a bounded subset $B$ of
$\mathcal A$, such that for each $x\in\mathcal A$ and
$\varepsilon>0$ and $\ell\in\mathbb{N}$,
there exists $b\in B$ such that $p_{\ell}(bx-x)<\varepsilon$ \;($p_{\ell}(xb-x)<\varepsilon$).\\
\item $\mathcal A$ has a left (right) multiple approximate
identity if given $\{a_1,...,a_n\}\subseteq\mathcal A$ and
$\varepsilon>0$ and $\ell\in\mathbb{N}$, we may find $b\in\mathcal
A$ such that
$$p_{\ell}(ba_i-a_i)<\varepsilon\;\;\;\;\;\;\;(p_{\ell}(a_ib-a_i)<\varepsilon),$$
for all $i=1,2,...,n$. We further say that $\mathcal A$ has a bounded left (right)
multiple approximate identity if there exists a bounded subset $B$
of $\mathcal A$, such that for each finite subset
$\{a_1,...,a_n\}$ of $\mathcal A$, we may find $b$ in $B$,
satisfying the above inequality. \item A left (right) approximate
identity in $\mathcal A$ is a net
$(e_{\alpha})_{\alpha\in\Lambda}$ in $\mathcal A$, such that
$$\lim_{\alpha}p_{\ell}(e_{\alpha}a-a)=0\;\;\;\;\;\;\;\;(\lim_{\alpha}p_{\ell}(ae_{\alpha}-a)=0),$$
for all $a\in\mathcal A$ and $\ell\in\mathbb{N}$. It is called a
bounded left (right) approximate identity if
$\{e_{\alpha},\;\;\alpha\in\Lambda\}$ is a bounded set in
${\mathcal A}$. We say that $(e_{\alpha})_{\alpha\in\Lambda}$ is
an approximate identity if it is both left and right approximate
identity.
\end{enumerate}
\end{DEF}

We commence with the following result, which is in fact a
generalization of \cite[Proposition 2 page 58]{Bon}, to
Fr$\acute{e}$chet algebras.

\begin{pro}
Let $(\mathcal A, p_{\ell})_{\ell\in {\Bbb N}}$ be a
Fr$\acute{e}$chet algebra. Then the following statements are
equivalent;
\begin{enumerate}
\item[(i)] $\mathcal A$ has bounded left approximate units,
\item[(ii)] $\mathcal A$ has a bounded left multiple approximate
identity, \item[(iii)] $\mathcal A$ has a  bounded left
approximate identity.
\end{enumerate}
\end{pro}

\begin{proof}
$(i)\Rightarrow (ii)$. By the hypothesis, there exists a bounded
subset $B$ of $\mathcal A$ such that for each $\ell\in {\Bbb N}$,
$\sup_{b\in B}p_{\ell}(b)\leq M_{\ell},$ for some $M_{\ell}>0$.
Moreover, for each $x\in\mathcal A$ and $\varepsilon>0$ and
$\ell\in\mathbb{N}$, there exists $b\in B$ with
$p_{\ell}(bx-x)<\varepsilon.$ Set $W=\{u+v-vu:\;\;u,v\in B\}.$ We
show that for each finite subset $F$ of $\mathcal A$ and
$\ell\in\mathbb{N}$ and $\varepsilon>0$, there exists $w\in W$
such that for all $x\in F$, $p_{\ell}(wx-x)<\varepsilon$. We prove
it inductively. First let $F=\{x_1,x_2\}$. So for each
$\ell\in\mathbb{N}$, there exist $u,v\in B$ such that
$$p_{\ell}(x_1-ux_1) <\frac{\varepsilon}{1+M_{\ell}}$$ and
$$p_{\ell}((x_2-ux_2)-v(x_2-ux_2))<\varepsilon.$$ Assuming
$w=u+v-vu$, we have $p_{\ell}(x_j-wx_j)<\varepsilon$, for $j=1,2$.
Now suppose that $(ii)$ holds for $\{x_1,...,x_n\}$ and consider
the finite subset $F=\{x_1,...,x_{n+1}\}$ of $\mathcal A$. Let
$\alpha_{\ell}=\max\{p_{\ell}(x_j):\;\;j=1,...,n\}.$ Thus for each
$\ell\in\mathbb{N}$ there exists $y\in W$ such that
$$p_{\ell}(x_j-yx_j)<\frac{\varepsilon}{3(1+M_{\ell})^2}\;\;,\;\;(j=1,2,...,n).$$
One can choose $w\in W$ such that
$$p_{\ell}(y-wy)<\frac{\varepsilon}
{3\alpha_{\ell}}\;\;\;\;and\;\;\;\;p_{\ell}(x_{n+1}-wx_{n+1})<\varepsilon.$$
Thus for each $j\in\{1,2,...,n\}$, we have
\begin{eqnarray*}
p_{\ell}(x_j-wx_j)&\leq& p_{\ell}(x_j-yx_j)+p_{\ell}(yx_j-wyx_j)+
p_{\ell}(wyx_j-wx_j)\\&\leq&
p_{\ell}(x_j-yx_j)+p_{\ell}(y-wy)p_{\ell}(x_j)+
p_{\ell}(w)p_{\ell}(yx_j-x_j)\\&<&\frac{\varepsilon}{3(1+M_{\ell})^2}+
\frac{\varepsilon}{3}+\frac{\varepsilon}{3(1+M_{\ell})^2}
(2M_{\ell}+M_{\ell}^2)<\varepsilon,
\end{eqnarray*}
which is our claim.

$(ii)\Rightarrow(i)$ is clear.\\
By \cite[Proposition $5.2$]{Pir}, $(ii)$ and $(iii)$, are
equivalent.
\end{proof}

The following result is interesting in its own right. It is a
generalization of a result due to Burnham \cite{Burn}, for Segal
Fr$\acute{e}$chet algebras. It shows that the condition (ii) in
Definition \ref{d1}, can be obtained from the conditions (i) and
(iii).

\begin{pro}\label{p3}
Let $(\mathcal A, p_{\ell})$ be a Fr$\acute{e}$chet algebra with a
right approximate identity denoted by $(e_\alpha)_\alpha$, and
$(\mathcal B, q_m)$ be a Fr$\acute{e}$chet algebra, such that
$\mathcal B$ is a dense left ideal in $(\mathcal A, p_{\ell})$.
Also suppose that the map
$$
(\mathcal A, p_{\ell})\times (\mathcal B, q_m)\longrightarrow
(\mathcal B, q_m),\;\;\;(a,b)\mapsto  ab,\;\;\;(a\in\mathcal A,
b\in\mathcal B)
$$
is continuous. Then $(\mathcal B, q_m)$ is a Segal
Fr$\acute{e}$chet algebra in $(\mathcal A, p_{\ell})$.
\end{pro}

\begin{proof}
It is sufficient to show that the map \eqref{e1} is continuous. We
will use Closed graph theorem \cite[Theorem 8.8]{ME} for
Fr$\acute{e}$chet algebras. Let $(b_n)$ be a sequence in $\mathcal
B$ such that $b_n\rightarrow 0$, in the topology of $\mathcal B$
and $b_n\rightarrow c$, for some $c\in\mathcal A$, in the topology
of $\mathcal A$. Since $(\mathcal B, q_m)$ is a Fr$\acute{e}$chet
algebra thus for every $b\in\mathcal B$, $b_nb\rightarrow 0$, in
the topology of $\mathcal B$. By the hypothesis, $b_nb\rightarrow
cb$, in the topology of $\mathcal B$. Thus $cb=0$, for all
$b\in\mathcal B$. Density of $\mathcal B$ in $\mathcal A$ implies
that $ca=0$, for all $a\in\mathcal A$. It follows that
$ce_\alpha=0$, for all $\alpha$. Since
$ce_\alpha\rightarrow_\alpha c$, in the topology of $\mathcal A$,
therefore $c=0$. Thus the result is obtained.
\end{proof}

\begin{cor}
Let $(\mathcal A, p_{\ell})$ be a Fr$\acute{e}$chet algebra with
an approximate identity, and $(\mathcal B, q_m)$ be a
Fr$\acute{e}$chet algebra, such that $\mathcal B$ is a dense
two-sided ideal in $(\mathcal A, p_{\ell})$. Also suppose that the
maps
$$
(\mathcal A, p_{\ell})\times (\mathcal B, q_m)\longrightarrow
(\mathcal B, q_m),\;\;\;(a,b)\mapsto  ab,\;\;\;(a\in\mathcal A,
b\in\mathcal B)
$$
and
$$
(\mathcal B, q_m)\times (\mathcal A, p_{\ell}) \longrightarrow
(\mathcal B, q_m),\;\;\;(b,a)\mapsto ba,\;\;\;(a\in\mathcal A,
b\in\mathcal B)
$$
are continuous. Then $(\mathcal B, q_m)$ is a symmetric Segal
Fr$\acute{e}$chet algebra in $(\mathcal A, p_{\ell})$.
\end{cor}

It is known that proper abstract Segal algebras never admit a
bounded approximate identity. In the following result we obtain
the same result for Segal Fr$\acute{e}$chet algebras.

\begin{pro}
Let $(\mathcal B, q_{m})_{m\in\mathbb{N}}$ be a proper Segal
Fr$\acute{e}$chet algebra in Fr$\acute{e}$chet algebra $(\mathcal
A, p_{\ell})_{\ell\in\mathbb{N}}$, with an approximate identity
$(e_{\alpha})_{\alpha\in \Lambda}$. Then $(e_{\alpha})_{\alpha\in
\Lambda}$ is not bounded in $\mathcal B$.
\end{pro}

\begin{proof}
Suppose on the contrary that $(e_{\alpha})_{\alpha\in \Lambda}$ is
bounded in $\mathcal B$. Thus for each $m\in\mathbb{N}$,
$\sup_{\alpha\in\Lambda}q_m(e_{\alpha})<K_m,$ for some $K_m>0$. On
the other hand by Remark \ref{r1}, for every $m\in\mathbb{N}$,
there exist $C_m>0$ and $\ell_{m}, n_m\in\mathbb{N}$ such that
$$q_m(ab)\leq C_m p_{{\ell}_{m}}(a)q_{n_m}(b),\;\;\;\;\;\;\;\;\;(a,b\in\mathcal B).$$
Since $(e_{\alpha})$ is a bounded right  approximate identity for
$\mathcal B$, for every $b\in {\mathcal B}$, $\varepsilon>0$ and
$m\in\mathbb{N}$, there exists $\alpha_{m}\in\Lambda$ such that
for each $\alpha\geq\alpha_{m}$, we have
$$
q_m(b)\leq q_m(be_{\alpha})+\varepsilon\leq C_m
p_{\ell_{m}}(b)q_{n_m} (e_{\alpha})+\varepsilon\leq C_m
K_{n_m}p_{\ell_{m}}(b)+\varepsilon.
$$
It follows that for each $b\in\mathcal B$,
$$
q_m(b)\leq C_m K_{n_m}p_{\ell_{m}}(b)+\varepsilon.
$$
Since $\mathcal B$ is dense in $\mathcal A$, it follows that
$\mathcal A=\mathcal B$, which is a contradiction.
\end{proof}

\begin{pro}\label{p1}
Let $(\mathcal A, p_{\ell})_{\ell\in\mathbb{N}}$ be a
Fr$\acute{e}$chet algebra and $\mathcal B$ be a dense subalgebra
of $\mathcal A$. Moreover suppose that $\mathcal B$ has a bounded
left approximate identity $(e_{\alpha})_{\alpha\in\Lambda}$. Then
$(e_{\alpha})_{\alpha\in\Lambda}$ is a bounded left approximate
identity for $\mathcal A$.
\end{pro}

\begin{proof}
Since $(e_{\alpha})_{\alpha\in\Lambda}$ is bounded in $\mathcal
A$, we have $p_{\ell}(e_{\alpha})\leq K_{\ell}$ for each
$\ell\in\mathbb{N}$ and some $K_{\ell}>0$. Clearly, we may assume
that $K_{\ell}\geq1$. Let $a\in \mathcal A$ and $\ell\in {\Bbb
N}$. Since $\mathcal B$ is dense in $\mathcal A$, for each
$\varepsilon>0$, there exists $b\in\mathcal B$ such that
$$p_{\ell}(b-a)<\frac{\varepsilon}{3K_{\ell}}.$$
On the other hand since $(e_{\alpha})_{\alpha\in\Lambda}$ is a
left approximate identity for $\mathcal B$, there exists
$\alpha_{\ell}\in\Lambda$ such that for each
$\alpha\geq\alpha_{\ell}$,
$$p_{\ell}(e_{\alpha}b-b)<\frac{\varepsilon}{3}.$$
Thus for each $\alpha\geq\alpha_{\ell}$ we have
\begin{eqnarray*}
p_{\ell}(e_{\alpha}a-a)&\leq&p_{\ell}(e_{\alpha}a-e_{\alpha}b)+p_{\ell}(e_{\alpha}b-b)+p_{\ell}(b-a)\\&\leq&
p_{\ell}(e_{\alpha})p_{\ell}(a-b)+p_{\ell}(e_{\alpha}b-b)+p_{\ell}(b-a)\\&\leq&
K_{\ell}p_{\ell}(a-b)+p_{\ell}(e_{\alpha}b-b)+p_{\ell}(b-a)\\&<&
K_{\ell}\frac{\varepsilon}{3K_{\ell}}+\frac{\varepsilon}{3}+\frac{\varepsilon}{3K_{\ell}}\leq\varepsilon.
\end{eqnarray*}
Consequently $(e_{\alpha})_{\alpha\in\Lambda}$ is a left bounded
approximate identity for $\mathcal A$.
\end{proof}

The following corollary is immediately obtained by Proposition
\ref{p1}.

\begin{cor}
Let $(\mathcal B, q_{m})_{m\in\mathbb{N}}$ be a Segal
Fr$\acute{e}$chet algebra in Fr$\acute{e}$chet algebra $(\mathcal
A, p_{\ell})_{\ell\in\mathbb{N}}$. Moreover suppose that $\mathcal
B$ has a left approximate identity
$(e_{\alpha})_{\alpha\in\Lambda}$ bounded in $\mathcal A$. Then
$(e_{\alpha})_{\alpha\in\Lambda}$ is a bounded left approximate
identity for $\mathcal A$.
\end{cor}

Closed left (right) ideals in symmetric abstract Segal algebras
have been characterized. In fact by ideal theorem, for every
symmetric Segal algebra $S(G)$, every closed left (right) ideal
$I$ of $S(G)$ is of the form $J\cap S(G)$, where $J$ is a unique
closed left (right) ideal of $L^1(G)$. We refer to \cite{f1} and
also \cite{re} for the basic definition of Segal algebras and also
all the required information about ideal theorem. Moreover this
theorem has been proved for abstract Segal algebras; see
\cite[Theorem 3.1]{Bar} and \cite[Theorem 3.2]{Bar}. As the final
result of the present work, we prove this result for Segal
Fr$\acute{e}$chet algebras.

\begin{thm}\label{t2}
Let $(\mathcal B,q_m)_{m\in {\Bbb N}}$ be a symmetric Segal
Fr$\acute{e}$chet algebra in Fr$\acute{e}$chet algebra $(\mathcal
A,p_{\ell})_{\ell\in\mathbb{N}}$. Then the following statements hold.
\begin{enumerate}
\item[(a)] If $J$ is a left ideal in $\mathcal A$, then
$cl_{\mathcal A}(J)$(closure of $J$ in $\mathcal A$) is a closed
left ideal in $\mathcal A$. \item[(b)] If $J$ is a left ideal in
$\mathcal A$, then $cl_{\mathcal A}(J)\cap \mathcal B$ is a closed
left ideal in $\mathcal B$. \item[(c)] If $I$ is a left ideal in
$\mathcal B$, then $cl_{\mathcal A}(I)$ is a closed left ideal in
$\mathcal A$. \item[(d)] If $I$ is a closed left ideal in
$\mathcal B$ and $\mathcal B$ has left approximate units, then
$I=cl_{\mathcal A}(I)\cap\mathcal B$.
\end{enumerate}
\end{thm}

\begin{proof}
$(a)$. Let $a\in\mathcal A$ and $b\in cl_{\mathcal A}(J)$. Then
there exists a sequence $(b_n)_{n\in\mathbb{N}}$ in $J$ such that
$\lim_{n\rightarrow\infty} p_{\ell}(b_n-b)=0$, for each
$\ell\in\mathbb{N}$. Since $p_{\ell}$ is submultiplicative, thus
$$p_{\ell}(ab_n-ab)\leq \;p_{\ell}(a)p_{\ell}(b_n-b).$$
Thus $\lim_{n\rightarrow\infty}ab_n=ab$, in the topology of
$\mathcal A$. Since $ab_n\in J$ for all $n\in\mathbb{N}$, it
follows that $ab\in cl_{\mathcal A}(J)$. This gives the
implication $(a)$.\\

$(b)$. By $(a)$ it is clear that $cl_{\mathcal A}(J)\cap \mathcal
B$ is a left ideal in $\mathcal B$. We prove that $cl_{\mathcal
A}(J)\cap \mathcal B$ is closed. Let $(a_n)_{n\in \mathbb{N}}$ be
a sequence in $cl_{\mathcal A}(J)\cap \mathcal B$ and
$a\in\mathcal B$ such that $\lim_{n\rightarrow\infty}a_n=a$, in
the topology of $\mathcal B$. Thus $\lim_{n\rightarrow
\infty}q_m(a_n-a)=0$, for every $m\in \mathbb{N}$. It follows that
$(a_n)_{n\in \mathbb{N}}$ tends to $a$, in the topology of
$\mathcal A$. Since $cl_{\mathcal A}(J)$ is closed in $\mathcal
A$, it follows that $a\in cl_{\mathcal A}(J)$. Consequently
$cl_{\mathcal A}(J)\cap\mathcal B$ is a closed left ideal in
$\mathcal B$.

$(c)$. Let $a\in \mathcal A$ and $b\in cl_{\mathcal A}(I)$. Hence
there exists a sequence $(b_n)_{n\in \mathbb{N}}$ in $I$ such that
$\lim_{n\rightarrow \infty }b_n=b$, in the topology of $\mathcal
A$. So $\lim_{n\rightarrow \infty}p_{\ell}(b_n-b)=0$, for each
$\ell\in \mathbb{N}$. Since $\mathcal B$ is dense in $\mathcal A$,
there exists a sequence $(c_n)_{n\in \mathbb{N}}$ in $\mathcal B$
such that $\lim_{n\rightarrow \infty}c_n=a$, in the topology of
$\mathcal A$. Consequently
$\lim_{n\rightarrow\infty}p_{\ell}(c_n-a)=0$, for each $\ell\in
\mathbb{N}$. Moreover we have
\begin{eqnarray*}
p_{\ell}(c_n b_n-ab)&\leq & p_{\ell}((c_n -a)b_n)+p_{\ell}(a(
b_n-b)))\\&\leq &
p_{\ell}(c_n-a)p_{\ell}(b_n)+p_{\ell}(b_n-b)p_{\ell}(a).
\end{eqnarray*}
Since $(b_n)_{n\in \mathbb{N}}$ is a bounded sequence in $\mathcal
A$, the right hand side of the above inequality tends to zero and
so $(c_n b_n)_{n\in \mathbb{N}}$ tends to $ab$, in the topology of
$\mathcal A$, which implies that $ab\in cl_{\mathcal A}(I)$. Thus
the result is obtained.

$(d)$. It is clear that $I\subseteq cl_{\mathcal A}(I)\cap
\mathcal B$. We prove the reverse of the inclusion. Suppose that
$a\in cl_{\mathcal A}(I)\cap \mathcal B$. For each $\varepsilon>0$
and $m\in \mathbb{N}$, there exists $u\in \mathcal B$ such that
$q_m(a-ua)<\varepsilon/2$. Since $a\in cl_{\mathcal A}(I)$, we can
find a sequence $(a_n)_{n\in \mathbb{N}}$ in $I$ such that
$\lim_{n\rightarrow \infty}a_n=a$, in the topology of $\mathcal
A$. It follows that $\lim_{n\rightarrow \infty}p_{\ell}(a_n-a)=0$,
for each $\ell\in \mathbb{N}$. Moreover there exist
$\ell_{m},n_m\in {\Bbb N}$ and $M_m>0$ such that
$$q_m(u a_n-ua)=q_m(u(a_n-a))\leq M_m p_{\ell_{m}}(a_n-a)q_{n_m}(u).$$
Consequently $\lim_{n\rightarrow \infty}q_m(ua_n-ua)=0$, for each
$m\in \mathbb{N}$. Thus there exists $N_m\in \mathbb{N}$, such
that $q_m(u a_{N_m}-ua)<\dfrac{\varepsilon}{2},$ and so $q_m(u
a_{N_m}-a)< \varepsilon$. Since $I$ is closed in $\mathcal B$,
thus $a\in I$. This completes the proof.
\end{proof}

\begin{cor}
Let $(\mathcal B,q_m)_{m\in {\Bbb N}}$ be a symmetric Segal
Fr$\acute{e}$chet algebra in Fr$\acute{e}$chet algebra $(\mathcal
A,p_{\ell})_{\ell\in\mathbb{N}}$ such that $(\mathcal B,q_m)_{m\in
{\Bbb N}}$ has left approximate units. If $I$ is a closed left
ideal in $\mathcal B$, then there exists a closed left ideal $J$
in $\mathcal A$ such that $I=J\cap\mathcal B$. In fact
$J=cl_{\mathcal A}(I)$.
\end{cor}

%{\bf Acknowledgment.} This research was partially supported by the
%center of excellence for mathematics at the University of Isfahan.

\vspace{9mm}

{\footnotesize \noindent
 F. Abtahi\\
  Department of Mathematics,
   University of Isfahan,
    Isfahan, Iran\\
     abtahif2002@yahoo.com\\

\noindent
 S. Rahnama\\
  Department of Mathematics,
   University of Isfahan,
    Isfahan, Iran\\
     rahnamasomaie51@gmail.com\\

\noindent
 A. Rejali\\
  Department of Mathematics,
   University of Isfahan,
    Isfahan, Iran\\
    alirejali12@gmail.com\\


\footnotesize


\begin{thebibliography}{99}

\bibitem{Bar} B. A. Barnes, {\it Banach algebras which are ideals in a Banach algebra}, Pacific J. Math., No.1 {\bf 38} (1971), 1-7.

\bibitem{Burn} J. T. Burnham, {\it Closed ideals in subalgebras of Banach algebras}, Proc. Amer. Math. Soc. No.2 {\bf 32} (1972), 551-555.

\bibitem{Burn2} J. T. Burnham, {\it Closed ideals in subalgebras of Banach algebras
II Ditkin's condition}, Monatsh. Math., {\bf 78}, (1974), 1-3.

\bibitem{Burn3} J. T. Burnham, {\it Segal algebras and dense ideals in Banach
algebras, Functional analysis and its applications ({I}nternat.
{C}onf., {E}leventh {A}nniversary of {M}atscience, {M}adras, 1973;
dedicated to {A}lladi {R}amakrishnan)}, Lecture Notes in Math.,
{\bf 399}, Springer, Berlin, 1974.

\bibitem{Bon} F. F. Bonsall and J. Duncan, Complete normed algebras, Springer- Verlag, Berlin, Hieldeberg, Newyork, 1973.

\bibitem{Cig} J.  Cigler,  {\it Normed  ideals  in $L^1(G)$}, Nederl.  Akad.
Wetensch.  Proc.  Ser.  A 72= Indag. Math.,  {\bf 31}, (1969),
273-282.

\bibitem{Dit}  V. A. Ditkin, {\it
Study of the structure of ideals in certain normed rings}, Ucenye
Zapinski Moskov Gos. Unive. Mathenmatika, {\bf 30}, (1939),
83-130.

\bibitem{D} D. H. Dunford, {\it Segal algebras and left normed ideals},
J. Lond. Math. Soc. (2), {\bf 8}, (1974), 514-516.

\bibitem{f1} H. G. Feichtinger, {\it Zur Idealtheorie von
Segal-algebren}, Manuscripta Math., {\bf 10}, (1973), 307-312.

\bibitem{f2} H. G. Feichtinger, {\it Results on Banach ideals and
spaces of multipliers}, Math. Scand. {\bf 41}, (1977), 315-324.

\bibitem{GSR} M. Gelfand, G. E. Shilov and D. A. Raikov, {\it
Commutative Normed Rings}, Chelsea Publishing Company, (1964)
[Bronx, New York] p.306.

\bibitem{Gold} H. Goldmann, {\it Uniform Fr$\acute{e}$chet algebras, North-Holland
Mathematics Studies}, {\bf 162}, North-Holand, Amesterdam-New
York, 1990.

\bibitem{Hel2} A. Ya. Helemskii, The homology of Banach and topological algebras (Moscow University Press,
English transl: Kluwer Academic Publishers, Dordrecht 1989).

\bibitem{Hel3} A. Ya. Helemskii, Banach and locally convex algebras, (Oxford Science Publications, (1993).

\bibitem{HR} Hewitt, E. and Ross, K. A., Abstract Harmonic analysis,
2nd edn. {\bf I}, Springer-Verlag, New York, (1970).

\bibitem{Law} P. Lawson and C. J. Read, Approximate amenability of Fr$\acute{e}$chet algebras, Math. Proc. Camb. Phil. Soc.,
{\bf 145}(2008), 403-418.

\bibitem{ME} R. Meise, D. Vogt, Introduction to functional analysis, (Oxford Science Publications), (1997).

\bibitem{Pir} A. Yu. Pirkovskii, {\it Flat cyclic Fr$\acute{e}$chet modules, amenable Fr$\acute{e}$chet algebras,
and approximate identities}, Homology, Homotopy Appl. {\bf 11},
no.1, (2009), 81-114.

\bibitem{re1} H. Reiter,
{\it Classical harmonic analysis and locally compact groups},
Clarendon, Press, Oxford, 1968.

\bibitem{re} H. Reiter,
{\it $L^1$-algebras and Segal Algebras}, Springer, Berlin,
Heidelberg, New York, 1971.

\bibitem{Ri} M. Riemersma, {\it On some properties of normed ideals in
$L^l(G)$}, Indag. Math., {\bf 37}, (1975), 265-272.

\bibitem{Shi} H. H. Schaefer, Topological vector spaces, Third printing corrected Graduate text in mathematics,
{\bf 3}, Springer-Velag, New York, 1971.

\bibitem{Sch} L. B. Schweitzer, {\it Dense nuclear Fr$\acute{e}$chet ideals in
$C^*-$algebras}, University of California, San Francisco, June
2013, preprint.

\bibitem{S} H. A. Smith, Tensor products of locally convex algebras, Proc. Amer. Math. Soc.,
{\bf 17}, (1966), 124-132.

\end{thebibliography}
\end{document}